\documentclass[12pt]{amsart}

\usepackage{cancel}
\usepackage{latexsym, amssymb, amsmath,esint}
\usepackage{soul}
\usepackage{amsfonts, graphicx}
\usepackage{graphicx,color}

\usepackage{amsmath, amssymb, amsfonts, mathrsfs, mathtools}
\usepackage{latexsym, esint, graphicx, xcolor, color}
\usepackage{wasysym, stmaryrd}
\usepackage{hyperref}
\usepackage[alphabetic]{amsrefs}

\newcommand{\be}{\begin{equation}}
\newcommand{\ee}{\end{equation}}
\newcommand{\beq}{\begin{eqnarray}}
\newcommand{\eeq}{\end{eqnarray}}

\usepackage{wasysym,stmaryrd}
\newtheorem{thm}{Theorem}[section]

\newtheorem{lma}[thm]{Lemma}

\newtheorem{prop}[thm]{Proposition}
\newtheorem{cor}[thm]{Corollary}

\theoremstyle{remark}
\newtheorem{rem}[thm]{Remark}

\numberwithin{equation}{section}

\newtheorem{claim}{Claim}[section]

\def\be{\begin{equation}}
\def\ee{\end{equation}}
\def\bee{\begin{equation*}}
\def\eee{\end{equation*}}

\def\blue{\color{blue}}

\def\K{K\"ahler }
\mathtoolsset{showonlyrefs}

\def\Ric{\text{\rm Ric}}

\def\tr{\operatorname{tr}}

\def\e{\varepsilon}

\def\a{{\alpha}}
\def\b{{\beta}}

\begin{document}

\title[Volume growth of non-negatively curved three-manifolds]{Volume growth and integral curvature bound for non-negatively curved three-manifolds }

\author{Pak-Yeung Chan}
\address[Pak-Yeung Chan]{Department of Mathematics, National Tsing Hua University, Hsin-Chu, Taiwan}
\email{pychan@math.nthu.edu.tw}

\author{Man-Chun Lee}
\address[Man-Chun Lee]{Department of Mathematics, The Chinese University of Hong Kong, Shatin, Hong Kong, China}
\email{mclee@math.cuhk.edu.hk}

\author{Mingxiang Li}
\address[Mingxiang Li]{Department  of Applied Mathematics, The Hong Kong Polytechnic  University,   Hong Kong, China}
\email{mingxiang.li@polyu.edu.hk}


\date{\today}

\begin{abstract}
Motivated by results in  K\"ahler geometry, in this work, we are interested in understanding the relation between integral curvature bounds and volume growth, under non-negative curvature in dimension three. In case of non-negative sectional curvature, we show that for metric on Euclidean space, it is of Euclidean volume growth if and only if it has average quadratic curvature decay. This is based on showing that metrics on three-dimensional Euclidean space with non-negative sectional curvature is of Euclidean volume growth if its asymptotic scaling invariant integral of scalar curvature is smaller than the sharp constant $8\pi$. We also show a gap Theorem if the curvature decay fast enough in the average sense, under non-negative Ricci curvature. 
\end{abstract}

\maketitle

\section{Introduction}
\label{sec: introduction}

Let $(M^n,g)$ be a complete non-compact manifold with non-negative curvature. The celebrated Theorem  by Cheeger-Gromoll-Meyer \cites{CheegerGromoll1972,GromollMeyer1969} said that if its sectional curvature $\mathrm{sec}(g)$ is non-negative, then $M$ is diffeomorphic to normal bundle of a compact totally convex, totally geodesic embedded sub-manifold $\mathcal{S}$, called the soul of $M$. When $n=3$, under the weaker assumption $\Ric(g)\geq 0$ it was shown by Liu \cite{Liu2013}, building on the work of Schoen-Yau \cite{SchoenYau}, that either $M^3$ is diffeomorphic to $\mathbb{R}^3$ or its universal cover $\tilde M=\Sigma^2\times\mathbb{R}$ splits. The higher dimensional analogy of differentiable structure under curvature conditions weaker than $\mathrm{sec}\geq 0$ is in general unclear. It is also conjectured by Topping \cite{ToppingPIC1} that non-negative $1$-isotropic curvature is the correct way to generalize $\Ric\geq 0$ for such a topological result.

Nevertheless, the differentiable structure under strong non-negativity curvature is generally well-understood. We are interested in the analytic aspect induced from non-negative curvature on non-compact manifolds. The underlying philosophy is that restriction of topological structure should also impose certain analytic constraint, particularly at spatial infinity. When $n=2$, it was first shown by Cohn–Vossen \cite{Cohn} that under $\mathrm{sec}\geq 0$, we must have $||\mathrm{scal}||_{L^1}\leq 4\pi$. In particular, its total curvature is finite. Later, Huber \cite{Huber} generalized this result under the assumption that the negative part of the scalar curvature is integrable, and proved that such open surfaces must be conformally equivalent to a compact surface with finitely many points removed.
In general dimension $n\geq 3$, it was generalized by Petrunin \cite{Petrunin2008} (see also \cite{Li})  
that a complete Riemannian manifold $(M^n,g)$ with $\mathrm{sec}\geq 0$ must satisfy
\begin{equation}\label{pet}
   \sup_{r>0}\;\;  r^{2-n}\int_{B_g(x_0,r)}\mathrm{scal}(g)\,d\mathrm{vol}_{g}\leq C_n,
\end{equation}
for all $x\in M$, where $C_n$ is a dimensional constant. Very recently, it was shown by Hao-Zhu \cite{HaoZhu2026} that the same conclusion cannot hold under $\Ric\geq 0$ and thereby resolving a problem of Yau \cite{Yau1992}, see also the independent work by Cheng \cite{Cheng2026} and Xu \cite{Xu2026}. See Yang's work \cite{Yang} for the counterexample for Yau's conjecture concerning  $k$-th elementary symmetric functions  of $\Ric$ with $k\geq 2$.  Building on the classification theorem \cites{CH,Zhu} for locally conformally flat manifolds with $\Ric\geq 0$, Ma \cite{Ma} showed \eqref{pet} holds in this special setting.  The examples in \cites{Cheng2026,HaoZhu2026,Xu2026} are volume collapsed at infinity, while the volume non-collapsed case is in general still open and is conjectured by Naber \cite{Naber}. We  refer readers to \cites{CucinottaMondino,Deng} for some progress in the non-collapsed case. 

Motivated by these, we are particularly interested in understanding the interplay between volume non-collapsing and integral curvature boundedness, under non-negative curvature. In the K\"ahler case, it has been studied extensively as part of the program in determining the complex structure of complete non-compact K\"ahler manifold with positive bisectional curvature $\mathrm{BK}>0$. In particular, the following is true by the joint effort of Ni, Ni-Tam and Liu: 
\begin{thm}[\cites{Ni2004,NiTam2013,Liu2016,Liu2018}]\label{thm:Kahler}
Suppose $(M^{n},g)$ is a complete non-compact \K manifold such that $\mathrm{dim}_{\mathbb{C}}(M)=n\geq 2$ and $\mathrm{BK}$ is quasi-positive, then the following are equivalent: 
\begin{enumerate}
    \item[(i)] $\displaystyle \mathrm{AVR}(g):=\lim_{r\to+\infty} \frac{\mathrm{Vol}_g(x_0,r)}{\omega_{2n} r^{2n}}>0$;
    \item[(ii)] There is non-trivial holomorphic function $f$ of polynomial growth;
    \item[(iii)] There is $C>0$ such that for all $x\in M$ and $r>0$,
    \begin{equation}\label{eqn:k}
k(x,r):=r^2\fint_{B_g(x,r)} \mathrm{scal}\,d\mathrm{vol}_g \leq C.
    \end{equation}
\end{enumerate}
\end{thm}

It naturally raise the question of whether (i) and (iii) in Theorem~\ref{thm:Kahler} are equivalent in any (purely) Riemannian content. The study is particularly active when $n=3$ under $\Ric\geq 0$. For instances, Xu \cite{Xu2024} proved that any complete manifold $(M^3,g)$ with $\Ric\ge 0$ and a pole must satisfy the sharp inequality
\begin{equation}\label{eqn:R-XU}
\lim_{r\to+\infty}\frac1r\int_{B_g(x_0,r)}\mathrm{scal}(g)\,d\mathrm{vol}_{g}=8\pi(1-\text{AVR(g)}),    
\end{equation}
which generalizes a previous estimate by Zhu \cite{ Zhu2022}. Assuming Euclidean volume growth but not the existence of pole, a version of \eqref{eqn:R-XU} weighted by the Green function, was recently obtained by Chen-Xu-Zhang \cite{CGZ2026}, building on the earlier work of Xu \cite{Xu2020}. We also refer interested readers to \cites{MW2025,Zhu2022a} for results toward generalizing Petrunin's estimate \cite{Petrunin2008}, as well as to \cite{Reiris2015} for a result establishing $\mathrm{AVR}(g)>0$ under pointwise quadratic curvature decay.

The first theme of this work is to  understand how (iii) implies (i), under $\mathrm{sec}\geq 0$. Our first main result is a volume gap Theorem of Petrunin's curvature estimate \cite{Petrunin2008}:
\begin{thm}\label{thm:main}
Suppose $g$ is a complete metric on $\mathbb{R}^3$ such that the sectional curvature $\mathrm{sec}(g)$ is non-negative. If there exists $x_0\in \mathbb{R}^3$ such that 
\begin{equation}
\a:=\limsup_{r\to+\infty}\frac1r\int_{B_g(x_0,r)}\mathrm{scal}(g)\,d\mathrm{vol}_{g}<6\omega_3:=8\pi
\end{equation}
then $(\mathbb{R}^3,g)$ is of Euclidean volume with 
\begin{equation}
    \mathrm{AVR}(g)=   1-\frac \a{8\pi}.
\end{equation}
\end{thm}

It is worth mentioning that the constant $8\pi$ is sharp even on $\mathbb{R}^3$. The rotationally symmetric metric with an asymptotically cylindrical end,
\begin{equation}\label{eqn:rot-eg}
    g:=dr^2+\tanh^2(r)g_{\mathbb{S}^2}
\end{equation}
 has $\mathrm{sec}\geq 0$ and $\mathrm{AVR}(g)=0$, whereas $\a=8\pi$. This in particular gives a three-dimensional analogy to Theorem~\ref{thm:Kahler} under $\mathrm{sec}\geq 0$:
\begin{cor}\label{cor:vol-decay-dich}
Suppose $(M^3,g)$ is a complete non-compact non-flat manifold such that $\mathrm{sec}\geq 0$ and  $\Ric(x_0)>0$ for some $x_0\in M$, then the following are equivalent: 
\begin{enumerate}
    \item[(i)] $\mathrm{AVR}(g)>0$;
      \item[(ii)] $\displaystyle \limsup_{r\to+\infty} r^2\fint_{B_g(x,r)} \mathrm{scal}(g)\,d\mathrm{vol}_g <+\infty$, for all $x\in M$.
\end{enumerate}
\end{cor}

Theorem~\ref{thm:main} is in other word saying that the curvature of a volume collapsed manifold $(\mathbb{R}^3,g)$ with $\mathrm{sec}\geq 0$ must be at least $8\pi$ in integral sense asymptotically, i.e. $\a\geq 8\pi$. Although the geometric quantity $\a$ is detecting non-flat at spatial infinity, it is at the same time influenced by local geometry. More generally, we show that the integral curvature grows linearly. The same also holds in general dimension $n\geq 3$: 
\begin{thm}\label{thm:scalar-growth}
Suppose $(M^n,g)$ is a complete non-compact manifold such that $\mathrm{sec}(g)\geq 0$ and $\mathrm{rank}(\Ric(x_0))\geq 3$ for some $x_0\in M$, then there exists $c_1>0$ such that for all $r>1$,
\begin{equation}
\frac1r    \int_{B_g(x_0,r)}\mathrm{scal}(g)\,d\mathrm{vol}_g \geq c_1.
\end{equation}
In particular, the total scalar curvature is infinite. 
\end{thm}

In fact, the constant $c_1$ in Theorem~\ref{thm:scalar-growth} can be explicitly controlled by curvature positivity at $x_0$, see Proposition~\ref{prop:c_1-local}. Moreover, the linear growth is sharp under $\mathrm{sec}>0$. This can be seen from the example above, i.e. the rotationally symmetric metric $g$ in \eqref{eqn:rot-eg} with $\mathbb{S}^2$ replaced by $\mathbb{S}^{n-1}$ for general $n\geq 3$. 

The second theme of this work is to explore the gap phenomenon of three-manifolds with $\Ric\geq 0$ and vanishing asymptotic curvature, i.e. $ k(x_0,r)$ in \eqref{eqn:k} is $o(1)$ as $r\to+\infty$. Despite the curvature positivity, this is the extreme case appeared in Corollary~\ref{cor:vol-decay-dich} where the geometric quantity in (ii) vanishes. Gap Theorem of this kind in the Riemannian content was studied by the first two named authors \cite{ChanLee2025}. The purpose there is to develop a Riemannian analogy of the optimal gap Theorem of Ni \cite{Ni2012} (see also \cite{NiNiu2020}) asserting that a complete non-compact \K manifold $(M^n,g)$ with $\mathrm{BK}\geq 0$ is flat if $ k(x_0,r)=0$ as $r\to 0$ for some $x_0\in M$. Our motivation is to explore the Riemannian counterpart of the Theorem of Ni \cite{Ni2012}, with the weakest possible curvature condition. In \cite{ChanLee2025}, the first two named authors used the Ricci flow approach to answer it in the non-collapsed case under non-negative $1$-isotropic curvature assumption\footnote{The combined methodologies in \cites{ChanLee2025,ChanLeePeachey} also implies the gap Theorem for $n=3$ in the non-collapsed case.}. And thus the remaining challenge is to treat the collapsing situation. When $n=3$, we exploit the potential theory of $p$-harmonic function and prove the optimal gap Theorem: 

\begin{thm}\label{thm:gap}
    Suppose $(M^3,g)$ is a complete non-compact manifold such that $\Ric\geq 0$ and there exists $x_0\in M$ such that 
    \begin{equation}\label{kto0}
        \limsup_{r\to+\infty}r^2\fint_{B_g(x_0,r)}\mathrm{scal}(g)\,d\mathrm{vol}_{g}=0,
    \end{equation}
    then $(M^3,g)$ is flat.
\end{thm}

Most results of this work relies on three dimensionality for the purpose of applying Gauss-Bonnet Theorem on various kind of level sets. In the first part, we study $(M^n,g)$ with $\mathrm{sec}\geq 0$ by exploiting the Busemann function $b:M\to\mathbb{R}$ and its properties from \cite{CheegerGromoll1972}. It is a $1$-Lipschitz function which is proper and convex, if $\mathrm{sec}\geq 0$. To address the volume growth in Theorem~\ref{thm:main}, we consider the sub-level set $C_\ell$ of the Busemann function $b$. Our topological assumption $M^3=\mathbb{R}^3$ ensure that the soul $\mathcal{S}$ of $M$ is a point so that $C_\ell$ are contractible. A very classical work \cite{Bujalo1978} of Bujalo, based on Gauss-Bonnet, asserts that the total scalar curvature on $C_\ell$ and the total (generalized) mean curvature on $\partial C_\ell$ can be controlled by the model on $\mathbb{R}^3$ sharply. This in turn control the volume growth of $M$ by variational formula of $\mathrm{Vol}_g(C_\ell)$ and distance comparison, modulus the error from total curvature decay $\a$. 

When we relax $\mathrm{sec}\geq 0$ to $\Ric\geq 0$, the busemann function $b$ is only sub-harmonic in general. Moreover when $n\geq 4$, it is in general not proper by the examples \cite{PanWei} of Pan-Wei, while the case $n=3$ is still open. We consider an alternative choice of exhaustion using Green type function. Under $\Ric\geq 0$, the properness of Green type function is tightened with its volume growth. In the setting of Theorem~\ref{thm:gap}, assume on the contrary $(M^3,g)$ is non-flat, the volume growth must be super-quadratic, thanks $\lim_{r\to+\infty} k(x_0,r)=o(1)$. For some well-chosen $p\in (1,3]$, this in particular implies $(M^3,g)$ is $p$-nonparabolic and its minimal $p$-Green function is proper. In reality, we consider instead $p$-capacitary function $u$ on some exterior domain $M\setminus B_g(x_0,\delta)$ and use its level set to exhaust $M\setminus B_g(x_0,\delta)$. Exploring the monotonicity developed by \cite{BQOP2024} with the help of gradient estimate of $p$-harmonic type function, we are able to show that $M$ must be of quadratic volume growth, contradicting to super-quadratic in the first place. The circle of idea also extends to a discussion of volume growth, in a similar spirit as Theorem~\ref{thm:main}. For a more general statement, we refer readers to Theorem~\ref{thm:vol-growth}.

\subsection*{Acknowledgments}
P.-Y. Chan is supported by the Yushan Young Fellow Program of the Ministry of Education (MOE) Taiwan (MOE-108-YSFMS-0004-012-P1), and by NSTC grants 113-2115-M-007 -014 -MY2 and 115-2628-M-007 -006. M.-C. Lee is supported by Hong Kong RGC grants No. 14300623 and No. 14304225, and an Asian Young Scientist Fellowship.

\subsection*{Disclosure on AI assistance.}
The idea of the proofs in this paper are due to ChatGPT 6.0 Pro, and the paper is an exposition of its output. The proofs has been verified and re-written by the authors and they take full responsibility for any errors.

\section{Properties of Busemann function } 

In this section, we collect some preliminary about Busemann function on manifolds with $\mathrm{sec}\geq 0$ which was studied extensively by Cheeger-Gromoll-Meyer \cites{CheegerGromoll1972,GromollMeyer1969}. Fix $x_0\in M$ and denote $\rho(x)=d_g(x_0,x)$.  We consider the set $  \mathcal{R}$ consisting of all rays $\gamma:[0,+\infty)\to M$ with $\gamma(0)=x_0$. The Busemann function $b:M\to \mathbb{R}$ is defined to be
\begin{equation}
    b(x):=\sup_{\gamma\in \mathcal{R}}\left\{  \lim_{s\to+\infty} \left(s-d_g(\gamma(s),x) \right)\right\}.
\end{equation}

By the celebrated work of  Cheeger-Gromoll-Meyer \cites{CheegerGromoll1972,GromollMeyer1969}, it is  known that $b$ is a convex function on $M$ and thus $\inf_M b<+\infty$. By translation, we will assume $\inf_M b=0$. We  consider the exhaustion of $M$ by its sub-level set. For all $\ell>0$, denote 
\begin{equation}
    C_\ell:=\{ x\in M: b(x)\leq \ell\}\subseteq M.
\end{equation}

The following properties will be important for us.
\begin{lma}\label{lma:property-buse}
For all $\ell>0$, we have 
\begin{enumerate}
    \item[(i)] Each $C_\ell$ is a totally convex compact set.
    \item[(ii)] $\mathrm{dim}(C_\ell)=n$ and $\cup_{\ell>0} C_\ell=M$.
    \item[(iii)] $C_\ell$ has the structure of an embedded sub-manifold of $M$ with smooth
totally geodesic interior and (possibly non-smooth) boundary.
\end{enumerate}
\end{lma}
\begin{proof}
It follows from \cite[Proposition 1.3 \& Theorem 1.6]{CheegerGromoll1972}.
\end{proof}

We also need the relation between $b(x)$ and $\rho(x)$.
\begin{lma}\label{lma:dist}
Under $\mathrm{sec}\geq 0$, the Busemann function $b$ satisfies the following: 
\begin{enumerate}
    \item[(i)] $b$ is $1$-Lipschitz function on $M$;
    \item[(ii)] There exists a function $\theta(s)$ with $\lim_{s\to+\infty}\theta(s)=0$ such that on $M$ we have
    \begin{equation}
        (1-\theta\circ \rho) \rho\leq b\leq \rho.
    \end{equation}
\end{enumerate}
\end{lma}
\begin{proof}
    The property (i) follows from definition and triangle inequality. The property (ii) follows from Toponogov’s comparison theorem, see \cite[Lemma 1]{Dress1994}.
\end{proof}

In particular, Lemma~\ref{lma:dist} provides us the lower bound for in-radius of each $C_\ell$. 
\begin{lma}\label{lma:in-rad}
The in-radius of $C_\ell$, defined by 
\begin{equation}
    t_0(C_\ell):=\sup\{ r>0: B(x,r)\Subset C_\ell\},
\end{equation}
satisfies $t_0(C_\ell)\geq \ell$.
\end{lma}
\begin{proof}
By Lemma~\ref{lma:dist}, $B(x_0,\ell)\subseteq C_\ell$ and hence $\ell\leq t_0(C_\ell)$. 
\end{proof}

\begin{lma}\label{lma:vol-ineq}
The Busemann function $b$ is differentiable a.e. and satisfies $|\nabla b|=1$ at differentiable point. Furthermore, 
\begin{equation}
    \mathrm{Vol}_g(C_\ell)=  \int^\ell_0 A(t)\,dt
\end{equation}
where $A(t):=\mathcal{H}^2(\partial C_t)$ for $t>0$.
\end{lma}
\begin{proof}
By Lemma~\ref{lma:dist} and Rademacher's theorem, the Busemann function $b$ is differentiable almost everywhere. At $x$ where $b$ is differentiable, we have $|\nabla b|\leq 1$, since $b$ is $1$-Lipschitz. It remains to prove the lower bound. We denote $b_\gamma(x):=\lim_{s\to+\infty} \left(s-d_g(\gamma(s),x) \right)$. We choose a sequence of ray $\{\gamma_j\}$ such that 
\begin{equation}\label{eqn:B-app}
    b_{\gamma_j}(x)\geq b(x)-j^{-1}.
\end{equation}

For each $\gamma_j$, take $t_i\to +\infty$ and consider the minimizing geodesic $\sigma_{i,j}(s)$ from $x$ to $\gamma_j(t_i)$. By compactness of geodesic, $\sigma_{i,j}$ sub-converges to geodesic $\sigma_j$. On the other hand for all $s\in [0,d_g(x,\gamma_j(t_i)]$,
\begin{equation}
\begin{split}
    \lim_{i\to+\infty}\left(t_i-d_g(\sigma_{i,j}(s),\gamma_j(t_i))\right)&=    s+\lim_{i\to+\infty}\left( t_i-d_g(\gamma_j(t_i),x)\right)\\
    &=s+b_{\gamma_j}(x)
\end{split}
\end{equation}
while the left hand side equals to $b_{\gamma_j}(\sigma_j(s))$. Therefore,
\begin{equation}
    b_{\gamma_j}(x)=b_{\gamma_j}(\sigma_{j}(s))-s
\end{equation}
for all $s\geq 0$. We also assume $\sigma_j$ sub-converges to geodesic $\sigma$ as $j\to+\infty$. This implies 
\begin{equation}
    \begin{split}
        b(\sigma(s))&=\lim_{j\to+\infty} b(\sigma_j(s))\\
        &\geq \lim_{j\to+\infty} b_{\gamma_j}(\sigma_j(s))\\
        &=\lim_{j\to+\infty} b_{\gamma_j}(x)+s\geq b(x)+s.
    \end{split}
\end{equation}
where we have used \eqref{eqn:B-app}. On the other hand since $b$ is $1$-Lipschitz,
\begin{equation}
    b(\sigma(s))-b(x)\leq d_g(\sigma(s),x)=d_g(\sigma(s),\sigma(0))=s.
\end{equation}
This shows that $b(\sigma(s))=b(x)+s$. If we let $v:=\sigma'(0)$, then 
\begin{equation}
    \begin{split}
        \nabla b|_x(v)=\frac{d}{ds}\Big|_{s=0} b(\sigma(s))=1.
    \end{split}
\end{equation}
This shows that $|\nabla b|(x)=1$. 

We now prove the second assertion. By co-area formula and $|\nabla b|=1$ a.e., 
    \begin{equation}
    \begin{split}
       \mathrm{Vol}_g(C_\ell)&=\int_{C_\ell} \,d\mathrm{vol}_g\\
       &= \int_{C_\ell} |\nabla b|\,d\mathrm{vol}_g=\int^\ell_0 A(t)\,dt.
           \end{split}
    \end{equation}
\end{proof}

\begin{lma}\label{lma:Gauss-Bonnet}
Suppose the soul is a point and $n=3$, then all $\ell>0$, we have $$\chi(\partial C_\ell)=2.$$
\end{lma}
\begin{proof}
By \cite[Theorems 2.1 and Theorem 2.5]{CheegerGromoll1972}, $C_s$ is homotopy equivalence to each other. 
    Since the soul is a point, it follows that $C_\ell$ is contractible and thus the result follows. 
\end{proof}

\section{Bujalo's convex body inequality in dimension three}\label{sec:convex-body}
In this section, we discuss the result of Bujalo \cite{Bujalo1978} on an inequality in controlling the integral of curvature by topological information. We first collect some materials used by Bujalo \cite{Bujalo1978}. We will keep the discussion minimal and refer readers to \cite{Bujalo1978} for the detailed exposition. We follow the discussion of \cite[Section 2]{Bujalo1978}. We assume $\mathrm{sec}\geq 0$ and $n=3$ throughout this section.

 For each closed locally convex set $C\Subset M^3$, by the work \cite{Walter1974} of Walter there is an open set $U$ containing $C$ such that:
\begin{enumerate}
    \item [(i)] For each $q\in U$, there is a unique $\Phi(q)\in C$, and a unique
minimal geodesic from $q$ to $\Phi(q)$ which lies entirely in $U$.
\item[(ii)] The function $f(q):=d_g(q,C)$ for $q\in U$, belongs to $C^{1,1}$. 
\end{enumerate}
In particular, $|\nabla f|=1$  on $U\setminus C$ and the map $\Phi$ is  given by 
$$\Phi(q)=\exp_q\left( -f(q)\nabla f(q)\right).$$

We consider $C_\tau:=\{q\in M: d_g(q,C)\leq \tau\}$ and the equidistance surface $F_\tau:=\{q\in M: d(q,C)=\tau\}$ where $F_\tau$ is a $C^{1.1}$-smooth hypersurface for all small $\tau>0$. We fix $N:=F_{\tau_0}$ for sufficiently small $\tau_0>0$ and consider the map $\varphi_\tau:=\Phi_\tau|_N:N\to F_\tau$ for $\tau\in (0,\tau_0]$, where 
\begin{equation}
    \Phi_\tau(q):=\exp_q\left( (\tau-f(q))\nabla f(q)\right).
\end{equation}

For $\tau\in (0,\tau_0]$, the map $\varphi_\tau$ is Lipschitz and hence differentiable a.e. on $N$ by Rademacher’s theorem. The Hausdorff measure of $\partial C$ is thus given by
\begin{equation}
    \mathcal{H}^2(\partial C)=\int_{N} I_0(p)\,dA_N,
\end{equation}
by \cite[Lemma 2.3.1]{Bujalo1978}, see also \cite[Theorem 3.2.3]{Federer1969}.  Furthermore, this can be recovered by approximation according to \cite[Lemma 2.3.2]{Bujalo1978}: 
\begin{equation}\label{lma:appr-meas}
    \mathcal{H}^2(\partial C)=\lim_{\tau\to 0^+} \int_{N} I_\tau(p)\,dA_N
\end{equation}
Here $I_\tau$ denotes the Jacobian of $\varphi_\tau$ which exists a.e. on $N$.

With this terminology, the mean curvature of $\partial C$ is defined through that of $F_\tau$. More precisely there exists $N'\subset N$ with $\mathcal{H}^{n-1}(N\setminus N')=0$ such that the second fundamental form $h_{\tau}$ is defined at $\varphi_\tau\in F_\tau$ for all $\tau\in (0,\tau_0]$. Moreover the mean curvature $\mathbf{H}(\tau,p)\cdot I_\tau(p)$ (weighted with its Jacobian) is uniformly bounded with finite limit
\begin{equation}
    \mathbf{H}(p):=\lim_{\tau\to 0^+} \mathbf{H}(\tau,p)\cdot I_\tau(p)\geq 0,
\end{equation}
from \cite[Lemma 2.4.6 \& Lemma 2.4.7]{Bujalo1978}. From \cite[Corollary 2.4.8]{Bujalo1978}, the total mean curvature $\mathcal{M}(\partial C)$ of $\partial C$ is then defined to be
\begin{equation}\label{eqn:defn-M-C}
    \mathcal{M}(\partial C):=\int_N \mathbf{H}(p)\, dA_N.
\end{equation}

\medskip
We are now ready to state the main inequality in this section.
\begin{prop}\label{prop:Bujalo-ineq}
For all $s>0$, we have
 $$4\pi \chi(\partial C_s) \cdot t_0(C_s)\leq \mathcal{M}(\partial C_s)+\int_{C_s} \mathrm{scal}\,\,d\mathrm{vol}$$
 where $t_0(C)$ denotes the in-radius of a set $C$, defined in Lemma~\ref{lma:in-rad}.
\end{prop}
\begin{proof}
Since $C_s$ is closed convex set by Lemma~\ref{lma:property-buse}, this follows from applying \cite[Theorem 5.2.1]{Bujalo1978} to the sublevel set $C_s$ of Busemann function.
\end{proof}

\begin{rem}
    If $C_s$ in Proposition~\ref{prop:Bujalo-ineq} has a smooth boundary $\partial C_s$, then $\mathcal{M}(\partial C_s)$ coincides with the total mean curvature of $\partial C_s$ by construction.
\end{rem}

\subsection{Minkowski inequality on convex body}\label{sec:Minkineql}

In this section, we will establish a radial Minkowski inequality relating the generalized total mean curvature of boundary of level set $C_s$ of Busemann function and its Hausdorff measure. 

\begin{thm}\label{thm:mink}
For every $s>0$, we have 
\begin{equation}
    s\cdot \mathcal{M}(\partial C_s)\leq 2\mathcal{H}^2(\partial C_s).
\end{equation}
\end{thm}

To make the idea more transparent, we split it into smooth case and general case. The general case requires the terminology introduced in Section~\ref{sec:convex-body}.

\subsubsection{Smooth case}
We first discuss the proof of Theorem~\ref{thm:mink} when all objects are smooth.

\begin{prop}\label{prop:thm-mink-smooth}
Theorem~\ref{thm:mink} holds when $\partial C_s$ is a smooth hypersurface. 
\end{prop}
\begin{proof}
Let $\nu$ be the outward unit normal to $\partial C_s$.  In our convention, the second fundamental form is given by $h(X,Y)=\langle \nabla_X \nu,Y\rangle$.  We consider the function $u:=\frac12\rho^2$. By Hessian comparison and $\mathrm{sec}\geq 0$, we have 
\begin{equation}
    \nabla^2 u\leq g
\end{equation}
in the sense of support functions and hence its restriction on $\Sigma=\partial C_s $, $u$ satisfies 
\begin{equation}\label{eqn:lap}
    \Delta_{\Sigma}u=\tr_{T\Sigma} \nabla^2 u-H\cdot \partial_\nu^- u\leq 2-H\cdot \partial_\nu^- u,
\end{equation}
in the sense of barrier, and thus distribution \cite[Appendix A]{MMU2014}. Here $\partial_\nu^-u$ is the least one-sided normal derivative determined by the minimizing geodesic from $p$.

By integrating \eqref{eqn:lap} over $\Sigma$, we conclude 
\begin{equation}\label{eqn:smooth-H-area}
    \begin{split}
        \int_\Sigma H\cdot \partial_\nu^- u\,\, d\mathrm{A}_\Sigma\leq 2\cdot\mathrm{Area}(\Sigma).
    \end{split}
\end{equation}

Since $H\geq 0$, it suffices to control $\partial_\nu u$ from below. Fix $x\in \Sigma$, we let $\gamma:[0,L]\to M$ be a minimizing geodesic from $x_0$ to $x$ with length $L$. Then 
\begin{equation}
    \begin{split}
        \partial_\nu^- u&=\langle \nu,\nabla u\rangle=\rho \langle \nabla b,\nabla \rho\rangle\\
        &=L \cdot \langle \nabla b,\gamma'(L)\rangle \\
        &=L\cdot \frac{d}{dt}\Big|_{t=L^-}( b\circ \gamma)(t).
    \end{split}
\end{equation}

If we denote $b\circ \gamma$ by $f$, then the convexity of $b$ implies $f$ is convex. Since $f(0)=0$ and $f(L)=s$, this implies $f'(L)\geq sL^{-1} $ and therefore,
\begin{equation}\label{eqn:u_nu-smooth}
\partial_\nu^- u=L\cdot \frac{d}{dt}\Big|_{t=L^-}f(t)\geq s.
\end{equation}

By substituting \eqref{eqn:u_nu-smooth} into \eqref{eqn:smooth-H-area}, this implies 
\begin{equation}
     s\cdot  \int_{\partial C_s} H\,d\mathrm{A}_{\partial C_s}\leq 2\cdot\mathrm{Area}(\partial C_s).
\end{equation}

This completes the proof.
\end{proof}

\subsubsection{General case}

We now prove Theorem~\ref{thm:mink} in full generality, i.e. when $\Sigma:=\partial C_s$ is still convex but is possibly non-smooth. 

\begin{proof}[Proof of Theorem~\ref{thm:mink}]

Following the discussion in section~\ref{sec:convex-body}, we consider 
\begin{equation}
    C_{s,\tau}:=\{q\in M: d_g(q,C_s)\leq\tau\},\;\;\text{and}\;\; \Sigma_{\tau}:=\{ q\in M: d_g(q,C_s)=\tau\}.
\end{equation}

We also denote $N:=\Sigma_{\tau_0}$ for sufficiently small $\tau_0>0$. For all $\tau\in (0,\tau_0]$, each $\Sigma_\tau$ is a closed $C^{1,1}$ hypersurface in $M$. We let $N'$ be the 
set where $\mathcal{H}^2(N\setminus N')=0$ where $h_\tau\circ\varphi_\tau$ is defined on $N'$. By \cite[Lemma 2.2.18 \& (2.2.2)]{Bujalo1978}, its second fundamental form $h_\tau\circ\varphi_\tau(p)$ satisfies 
\begin{equation}
-C_0\tau\leq     h_\tau\circ \varphi_\tau \leq C_1\tau^{-1}.
\end{equation}

By \cite[Lemma 5.6]{CRW2015}, we might further regularize each $C^{1,1}$ hypersurface $\Sigma_\tau$ by smooth hypersurfaces $\Sigma_{\tau,j}$ such that $\Sigma_{\tau,j}$  converges to $\Sigma_\tau$ in $C^1$ topology and satisfies 
\begin{equation}\label{eqn:h-bdd}
-C_0\tau-j^{-1}\leq     h_{\Sigma_{\tau,j}} \leq C_1\tau^{-1}+j^{-1}.
\end{equation}

Since $\Sigma_{\tau,j}$ is smooth, we might apply \eqref{eqn:smooth-H-area} from the proof of Proposition~\ref{prop:thm-mink-smooth} to deduce 
\begin{equation}\label{eqn:appr-}
    \begin{split}
        \int_{\Sigma_{\tau,j}} H_{\Sigma_{\tau,j}}\cdot \partial_\nu^- u\,d\mathrm{A}_{\Sigma_{\tau,j}}\leq 2\cdot\mathrm{Area}(\Sigma_{\tau,j}).
    \end{split}
\end{equation}
for all $\tau\in (0,\tau_0]$ and $j\to+\infty$. Using \eqref{eqn:h-bdd} and $|\nabla u|\leq C_1$, 
\begin{equation}
    \begin{split}
\inf_{\Sigma_{\tau,j}}\partial_\nu^-u \cdot \int_{\Sigma_{\tau,j}} H_{\Sigma_{\tau,j}}  d\mathrm{A}_{\Sigma_{\tau,j}}
    &\leq \left( 2+C_2(\tau+j^{-1})\right)\cdot\mathrm{Area}(\Sigma_{\tau,j})
    \end{split}
\end{equation}

By the two-sided bound of second fundamental form \eqref{eqn:h-bdd}, we might pass $j\to+\infty$ to conclude 
\begin{equation}\label{eqn:H-u-seq}
    \begin{split}
\inf_{\Sigma_{\tau}}\partial_\nu^-u \cdot \int_{N} \mathbf{H}(\tau,p) \cdot I_\tau(p) d\mathrm{A}_{N}
    &\leq \left( 2+C_2\tau\right)\cdot\mathrm{Area}(\Sigma_{\tau})
    \end{split}
\end{equation}
where $\inf_{\Sigma_{\tau}}\partial_\nu^-u$ denotes the infin-mum of least one-sided normal derivatives determined by the minimizing geodesic from $x_0$ to $x\in \Sigma_\tau$.

We next claim that the normal derivatives $\partial_\nu^- u$ is asymptotically bounded from below by $s$ as in the smooth case.
\begin{claim}\label{claim:lower-bdd-dini}
As $\tau\to 0^+$, we have 
\begin{equation}
   \liminf_{\tau\to 0^+} \left(\inf_{\Sigma_{\tau}}\partial_\nu^-u\right)\geq s
\end{equation}
\end{claim}
\begin{proof}[Proof of Claim~\ref{claim:lower-bdd-dini}] 
Suppose on the contrary, there is $\delta>0$, $\tau_i\to 0^+$, $x_i\in \Sigma_{\tau_i}$, outer normal $\nu_i$ and a minimizing from $x_0$ to $x_i$ whose terminal unit vector $v_i$ satisfies 
\begin{equation}
    d_g(x_0,x_i)\cdot \langle \nu_i,v_i\rangle\leq s-\delta.
\end{equation}

Denote $y_i:=\varphi_{\tau_i}(x_i)$. Then compactness implies $x_i,y_i\to x_\infty \in \partial C_s$, $\nu_i\to \nu_\infty$ and $v_i\to v_\infty$ where $\nu$ is an outward unit support normal to $C_s$ at $x_\infty$ and $v_\infty$ is the terminal direction of a minimizing geodesic $\gamma_\infty$ from $x_0=\gamma_\infty(0)$ to $x_\infty=\gamma_\infty(L)$ so that 
\begin{equation}\label{eqn:limit-contra}
     L\cdot \langle \nu_\infty,\gamma_\infty'(L)\rangle\leq s-\delta.
\end{equation}

At $x_\infty$, for some $\lambda\in (0,1]$ $\lambda\nu_\infty$ is a sub-gradient and thus $\lambda \langle \nu_\infty,\gamma_\infty'(L)\rangle \geq sL^{-1}$ which contradicts with \eqref{eqn:limit-contra}. This proves the claim.
\end{proof}

By Claim~\ref{claim:lower-bdd-dini}, \eqref{eqn:H-u-seq} and \eqref{eqn:defn-M-C}, we might let $\tau\to0$ to conclude Theorem~\ref{thm:mink}.
\end{proof}

\section{Proof of Theorem~\ref{thm:main}}

In this section, we show that a universal bound $8\pi$ of the average integral curvature
$$\hat k(x_0,r):=r^{-1}\int_{B_g(x_0,r)}\mathrm{scal}\,d\mathrm{vol}_g$$ will force $M^3$ to be non-collapsed at infinity. We isolated one direction of inequality in case of maximal volume growth, which holds under only $\Ric\geq 0$.
\begin{prop}\label{prop:MV-inte-R}
Suppose $(M^3,g)$ is a complete non-compact manifold such that $\Ric\geq 0$ and $\mathrm{AVR}(g)>0$, then for any $x_0\in M$, 
\begin{equation}
    \liminf_{r\to+\infty} \frac{1}{r}\int_{B_g(x_0,r)} \mathrm{scal}\,d\mathrm{vol}_g \geq 8\pi \left( 1-\mathrm{AVR}(g)\right)\geq 0.
\end{equation}
\end{prop}
\begin{proof}
    
Let $u:=(4\pi G_{x_0})^{-1}$ and $G_{x_0}$ is the Green function with pole $x_0$, then \cite[Theorem 1.5]{CGZ2026} gives us 
\begin{equation}\label{eqn:Xu-estimate}
\left\{
\begin{array}{ll}
\lim_{t\to+\infty}t^{-1} J(t) = 8\pi(1-\mathrm{AVR}(g));\\[3mm]
 \displaystyle J(t):=\int_{\{u\leq t\}}\mathrm{scal}\cdot |\nabla u|\,d\mathrm{vol}_g.
\end{array}
\right.
\end{equation}
\medskip

On the other hand, \cite[Theorem 3.26]{Colding} implies 
\begin{equation}
    \lim_{r\to+\infty} \sup_{M\setminus B_g(x_0,r)} |\nabla u|=\mathrm{AVR}(g).
\end{equation}
Together with \cite[(3.37)]{ColdingMinicozzi}, for $\e>0$, we choose $s_\e>0$ such that outside $B_g(x_0,s_\e)$,
\begin{equation}
    u(x)\geq (\mathrm{AVR}-\e)d_g(x,x_0),\;\; |\nabla u|\leq \mathrm{AVR}+\e.
\end{equation}
Then $\{u<(\mathrm{AVR}-\e)r\}\subset B_g(x_0,r)$ for all large $r$ and hence, 
\begin{equation}\label{eqn:J}
\begin{split}
r^{-1}J\left( (\mathrm{AVR}-\e)r\right)\leq C_\e r^{-1}+ (\mathrm{AVR}+\e)r^{-1}\int_{B_g(x_0,r)}\mathrm{scal}\,d\mathrm{vol}_g.
\end{split}
\end{equation}

Result follows by letting $r\to+\infty$ on \eqref{eqn:J} using also \eqref{eqn:Xu-estimate} and followed by $\e\to 0$.
\end{proof}

We now proceed to prove the opposite direction. 
\begin{proof}[Proof of Theorem~\ref{thm:main}] 
By Soul Theorem, the soul $\mathcal{S}$ is a point as the manifold is topologically Euclidean. We fix $x_0\in M$ and consider the Busemann function $b$ on $M$. 

By Lemma~\ref{lma:dist}, for any $\e>0$, there exists $R_\e>0$ such that $(1+\e)^{-1}\rho\leq b$ for all $\rho>R_\e$. In particular for all $s$ large, $
    C_s\subseteq B_g(x_0, (1+\e)s)$,
and hence
\begin{equation}\label{eqn:vol-low}
   \int^s_0 A(t)\,dt \leq  \mathrm{Vol}_g(C_s)\leq \mathrm{Vol}_g\left( x_0,(1+\e)s \right)
\end{equation}
by Lemma~\ref{lma:vol-ineq}, where $A(t):=\mathcal{H}^2(\partial C_t)$. We also assume $R_\e$ is large enough such that
\begin{equation}\label{eqn:alpha-bdd}
    \frac1r \int_{B_g(x_0,r)}\mathrm{scal}\, d\mathrm{vol}_g< \a+\e <8\pi
\end{equation}
for all $r>R_\e$.

By Lemma~\ref{lma:in-rad}, Lemma~\ref{lma:Gauss-Bonnet} and Proposition~\ref{prop:Bujalo-ineq},  we have 
\begin{equation}
    \begin{split}
      8\pi t &\leq  4\pi\chi(\partial C_t) \cdot t_0(C_t)\\
      &\leq \mathcal{M}(\partial C_t)+\int_{C_t}\mathrm{scal}\,d\mathrm{vol}_g
    \end{split}
\end{equation}

Hence Theorem~\ref{thm:mink} implies
\begin{equation}\label{eqn:loww}
    \begin{split}
 4\pi t^2-\frac12 t\int_{C_t}\mathrm{scal}\,d\mathrm{vol}_g    \leq \frac12 t\cdot \mathcal{M}(\partial C_t)   \leq A(t)
    \end{split}
\end{equation}

Combines \eqref{eqn:loww} with \eqref{eqn:vol-low}, we conclude 
\begin{equation}\label{eqn:VOLL}
    \begin{split}
    \omega_3 s^3 \leq   \mathrm{Vol}_g\left( x_0,(1+\e)s \right)+\frac12 \int^s_0 \int_{C_t}t\cdot \mathrm{scal}\,d\mathrm{vol}_g
    \end{split}
\end{equation}
for all $s$ large. 

For the last integral, we split it into 
\begin{equation}
    \begin{split}
  \frac12 \int^s_0 \int_{C_t}t\cdot \mathrm{scal}\,d\mathrm{vol}_g&=\frac12 \left(\int^s_{R_\e}+\int_0^{R_\e}\right) \int_{C_t}t\cdot \mathrm{scal}\,d\mathrm{vol}_g\\
    &=:\mathbf{I}+\mathbf{II}
    \end{split}
\end{equation}
where $\mathbf{II}=o(s^3)$ for each fixed $\e>0$. For $\mathbf{I}$, we use \eqref{eqn:alpha-bdd} so that 
\begin{equation}
    \begin{split}
        \mathbf{I}&\leq \frac1{2} \int^s_{R_\e}  t\,\int_{B_g(x_0,(1+\e)t)}\mathrm{scal}\,d\mathrm{vol}_g\\
        &\leq  \frac1{2} (\a+\e) (1+\e)   \int^s_{R_\e} t^2\,dt\leq \frac1{6} (\a+\e) (1+\e)   s^3.
    \end{split}
\end{equation}

And hence \eqref{eqn:VOLL} is reduced to 
\begin{equation}
    1\leq \frac{\mathrm{Vol}_g\left( x_0,(1+\e)s \right)}{\omega_3s^3}+\frac1{6\omega_3} (\a+\e) (1+\e)   +o(1)
\end{equation}
for all $s\to+\infty$. By volume comparison, we pass $s\to+\infty$ and followed by $\e\to 0$ to show that
\begin{equation}
 1-\frac \a{6\omega_3}\leq \mathrm{AVR}(g):=\lim_{s\to+\infty}\frac{\mathrm{Vol}_g\left( x_0, s \right)}{\omega_3s^3}.
\end{equation}
The equality follows from Proposition~\ref{prop:MV-inte-R}.
\end{proof}

The dichotomy analogy of Theorem~\ref{thm:Kahler} is now immediate. 
\begin{proof}[Proof of Corollary~\ref{cor:vol-decay-dich}]
If (i) is true, then (ii) follows directly from Petrunin \cite{Petrunin2008}. The difficult part is to prove the opposite direction. By the work of Schoen-Yau \cite{SchoenYau}, $M^3$ is diffeomorphic to $\mathbb{R}^3$. Suppose $\mathrm{AVR}(g)=0$, then Bishop-Gromov volume comparison yields 
\begin{equation}
    \begin{split}
    \frac1r\int_{B_g(x_0,r)} \mathrm{scal}\,d\mathrm{vol}_g 
    &\leq \frac{\mathrm{Vol}_g(x_0,r)}{r^3}\cdot r^2\fint_{B_g(x_0,r)} \mathrm{scal}\,d\mathrm{vol}_g.
    \end{split}
\end{equation}

By letting $r\to+\infty$, this shows that $\a=0$ in Theorem~\ref{thm:main} and hence $\mathrm{AVR}>0$. This is impossible.
\end{proof}

\section{Curvature growth of non-flat manifold }

In Theorem~\ref{thm:main}, we see how integral curvature bound give rise to non-collapsing. We now show that if in addition $\Ric(x_0)>0$ somewhere, then $\hat k(x_0,r)$ growth at least linearly. Indeed, we will prove the estimates for general $n\geq 3$.

\begin{lma}\label{lma:T}
Suppose $\mathrm{sec}\geq0$, then the tensor $$T:=\frac12\mathrm{scal}\cdot g-\Ric$$ is non-negative and is divergent free.
\end{lma}
\begin{proof}
This follows from \cite[Lemma 3.2]{ChanLee2025} and Bianchi identity.
\end{proof}

To ease the regularity issue caused in differentiating the Busemann function $b$, we use the approximation method of Greene-Wu \cite{GreeneWu}:
\begin{lma}\label{lma:GreeneWu}
There exists $\e_i\to 0^+$ and a sequence of smooth function $b_i\in C^\infty(M)$ such that 
\begin{enumerate}
    \item $|b_i-b|\leq \e_i$;
    \item $|\nabla b_i|\leq 1$;
    \item $\nabla^2 b_i\geq -\e_i g$ on $M$.
\end{enumerate}
\end{lma}
\begin{proof}
It follows from the convexity of $b$, \cite[Proposition 2.3]{GreeneWu} and Lemma~\ref{lma:dist}.
\end{proof}

Now we use $b_i$ to approximate $b_\infty:=b$ and consider a signed measure on $M$ as follows. We let $\phi(s):=s+\sqrt{1+s^2}$, $f_i:=\phi(b_i)$ and $X_i:=T^\sharp \nabla f_i$. 

\begin{lma}\label{lma:almost-nonneg-measure}
The function $\mathrm{div}(X_i)$ satisfies 
\begin{equation}
    \mathrm{div}(X_i)\geq -o_{loc}(1), \;\text{as}\;\;i\to+\infty.
\end{equation}
\end{lma}
\begin{proof}
By direct computation and Lemma~\ref{lma:T}, 
    \begin{equation}\label{eqn:div-almostnonnega}
        \begin{split}
            \mathrm{div}(X_i)&=\langle T,\nabla^2f_i\rangle\\
            &=\phi'(b_i)\cdot  \langle T,\nabla^2 b_i\rangle+\phi''(b_i) \cdot T(\nabla b_i,\nabla b_i)\\
            &\geq -\e_i\left( \frac{n}{2}-1\right)  \cdot \phi(b_i) \cdot \mathrm{scal}+\phi''(b_i) \cdot T(\nabla b_i,\nabla b_i)\\
            &\geq -\e_i\left( \frac{n}{2}-1\right)  \cdot \phi(b_i) \cdot \mathrm{scal},
        \end{split}
    \end{equation}
where we have used Lemma~\ref{lma:GreeneWu}, Lemma~\ref{lma:T} and $\tr_g T=(\frac{n}2-1)\mathrm{scal}$. This completes the proof by Lemma~\ref{lma:GreeneWu}.
\end{proof}

\medskip
\begin{prop}\label{prop:non-vanishmeasure}
Suppose $(M,g)$ is a complete non-compact manifold such that $\mathrm{sec}\geq 0$ and $\mathrm{rank}(\Ric(x_0))\geq 3$ for some $x_0\in M$, then for any $r>0$ we have 
$$\displaystyle\liminf_{i\to+\infty} \int_{B_g(x_0,r)} \mathrm{div}(X_i)\,d\mathrm{vol}_g>0.$$
\end{prop}
\begin{proof}
Suppose the conclusion fails, by Lemma~\ref{lma:almost-nonneg-measure} we might assume
\begin{equation}\label{eqn:divv-seq}
    \lim_{i\to+\infty} \int_{B_g(x_0,r)} \mathrm{div}(X_i)\,d\mathrm{vol}_g=0,
\end{equation}
after passing to subsequence.

By \eqref{eqn:div-almostnonnega}, Lemma~\ref{lma:T} and Lemma~\ref{lma:GreeneWu},  
\begin{equation}\label{eqn:L1-T}
 \lim_{i\to+\infty}   \int_{B_g(x_0,r)}  T(\nabla b_i,\nabla b_i)\,d\mathrm{vol}_g=0.
\end{equation}

By Lemma~\ref{lma:vol-ineq}, we let $U\subseteq B_g(x_0,r)$ be a subset such that $b=b_\infty$ is differentiable on $U$ and $\mathrm{Vol}_g\left(B_g(x_0,r)\setminus U \right)=0$. Using \eqref{eqn:L1-T},  Lemma~\ref{lma:T}, full measure of $U$ and semi-continuity of $\mathrm{rank}(\Ric)$, we choose $x_1\in U$ so that 
\begin{enumerate}
    \item[(a)]$\lim_{i\to+\infty}T(\nabla b_i,\nabla b_i)(x_1)=0$;
    \item[(b)] $\mathrm{rank}(\Ric(x_1))\geq 3$.
\end{enumerate}

At $x_1$, $\nabla b_i\in T_{x_1}M$ such that $|\nabla b_i|(x_1)\leq 1$ by Lemma~\ref{lma:GreeneWu}. 
By passing to subsequence, we might assume $\nabla b_i(x_1)\to Y$ as $i\to+\infty$, for some vector $Y\in T_{x_1}M$. In particular, either $Y=0$ or $Y\in \mathrm{Ker}(T(x_1))$.  On the other hand, using the almost convexity of $b_i$ from Lemma~\ref{lma:GreeneWu} and $C^0$ convergence of $b_i$ to $b_\infty$, we conclude that for any geodesic $\gamma$ starting from $\gamma(0)=x_1$, we have 
\begin{equation}
b(\gamma(t))\geq   \langle Y,\gamma'(0) \rangle \cdot t +b(x_1).
\end{equation}
In particular, $Y$ is a sub-gradient of $b$ at $x_1$. Since we have assumed $b$ is differentiable at $x_1$, $Y=\nabla b(x_1)$ and hence $Y\neq0$ since $|\nabla b|(x_1)=1$ by Lemma~\ref{lma:vol-ineq}.

\begin{claim}\label{claim:null}
We have $\mathrm{Null}(T(x_1))\leq 1$.
\end{claim}
\begin{proof}[Proof of Claim~\ref{claim:null}]
    Suppose on the contrary, there is an orthonormal frame $\{e_i\}_{i=1}^n$ in $T_{x_1}M$ such that $T(e_1)=T(e_2)=0$. Then we have $R_{11}=R_{22}=\frac12 \mathrm{scal}$ and hence $\mathrm{rank}(\Ric)=2$ which is impossible.
\end{proof}

By Claim~\ref{claim:null}, we might choose a smooth vector field $V$ on a neighborhood $W$ of $x_1$ such that $\mathrm{Ker}(T)=\mathrm{span}\{V\}$ on $W$ and $|V|=1$. We also choose $V$ so that $V(x_1)=Y=\nabla b(x_1)\in T_{x_1}M$. 

\begin{claim}\label{claim:extens}
There exists an open neighborhood $W'$ of $x_1$ such that $$\nabla b(x)=V(x)$$ on $W'\cap U$. In particular, $b$ is smooth on $W'$. 
\end{claim}
\begin{proof}[Proof of Claim~\ref{claim:extens}]
Suppose on the contrary, there exists a sequence $\{z_k\}_{k=1}^\infty$ such that $z_k\in U$, $z_k\to x_1$ and $\nabla b(z_k)\neq V(z_k)$ for all $k\to+\infty$. By repeating the above argument to $z_k$ in place of $x_1$, we must have $\nabla b(z_k)=-V(z_k)$ such that $\nabla b(z_k)\to -V(x_1)$ by continuity of $V$. But this contradicts with the continuity of sub-gradient at convergent differentiable point. This shows that $\nabla b=V$ almost everywhere on some open neighborhood $W'$. Since $b$ is locally Lipchitz with smooth weak derivatives, $b$ is smooth there.
\end{proof}

The Claim~\ref{claim:extens} also shows the first part of $(\mathrm{II})$ and thus $\mathrm{div}(X_\infty)$ exists smoothly on $W'$. By \eqref{eqn:div-almostnonnega} and smoothness of $b$, we want to rule out the situation that on $W'$,
\begin{enumerate}
    \item $T(\nabla b)\equiv 0$ and
    \item $\langle T,\nabla^2 b\rangle\equiv 0$. 
\end{enumerate}

Since $\nabla b$ is non-trivial from Lemma~\ref{lma:vol-ineq}, $T$ must be of rank $n-1$ on $W'$ using Claim~\ref{claim:null}. From $\langle T,\nabla^2 b\rangle=0$, $\nabla^2 b=\lambda V\otimes V$ on $W'$ where $\lambda(x)\geq0$. We now differentiate $|\nabla b|^2=1$ to see that $\langle \nabla^2 b,\nabla b\rangle=0$ and thus $\nabla b$ is a parallel vector field on $W'$. 

By Bochner formula, we therefore conclude that $\Ric(\nabla b,\nabla b)=0$, while 
\begin{equation}
    0=\Ric(\nabla b,\nabla b)=\frac12 \mathrm{scal} \cdot |\nabla b|^2>0,
\end{equation}
since $\mathrm{rank}(\Ric)\geq 1$ and $\Ric\geq 0$ on $W'$. This is impossible. In conclusion, we have shown that 
\begin{equation}
    \int_{W'}\mathrm{div}(X_\infty)\,d\mathrm{vol}_g >0.
\end{equation}

Finally we claim that this will contradict with \eqref{eqn:divv-seq}. We fix a smooth non-negative function $\varphi$ on $M$ such that $\varphi=1$ on $W''\Subset W'$ and vanish outside $W'$. By Stoke Theorem for all $i\in\mathbb{N}\cup\{+\infty\}$ and divergent free from Lemma~\ref{lma:T}, we have
\begin{equation}\label{eqn:inte-with err-0}
\begin{split}
\int_{M} \mathrm{div}(X_i)\cdot\varphi\,d\mathrm{vol}_g&=-\int_M \langle X_i,\nabla \varphi\rangle \,d\mathrm{vol}_g\\
&=-\int_M   T( \nabla f_i,\nabla \varphi) \,d\mathrm{vol}_g\\
&=\int_M   f_i \cdot \langle T,\nabla^2 \varphi\rangle\,d\mathrm{vol}_g.
    \end{split}
\end{equation}

Thanks to \eqref{eqn:div-almostnonnega} and Lemma~\ref{lma:GreeneWu}, 
\begin{equation}
    \begin{split}
 0\quad&< \int_{M} \mathrm{div}(X_\infty)\cdot\varphi\,d\mathrm{vol}_g\\
 &=\int_M   f_\infty \cdot \langle T,\nabla^2 \varphi\rangle\,d\mathrm{vol}_g\\
 &=\lim_{i\to+\infty}\int_M   f_i \cdot \langle T,\nabla^2 \varphi\rangle\,d\mathrm{vol}_g\\
 &=\lim_{i\to+\infty}\int_{M} \mathrm{div}(X_i)\cdot\varphi\,d\mathrm{vol}_g.
    \end{split}
\end{equation}

Since $W'\subseteq B_g(x_0,r)$, this contradicts with \eqref{eqn:divv-seq} and thus finishing the proof.
\end{proof}

\medskip

\begin{proof}[Proof of Theorem~\ref{thm:scalar-growth}]

For all $r>1$, we let $\varphi$ be a smooth non-negative function on $M$ such that $\varphi=1$ on $B_g(x_0,r)$, vanishes outside $B_g(x_0,2r)$ and satisfies $|\nabla \varphi|\leq 2r^{-1}$ on $M$. 

By Proposition~\ref{prop:non-vanishmeasure}, we have 
\begin{equation}\label{eqn:choice-c_0}
    c_0:=\liminf_{i\to+\infty} \int_{B_g(x_0,1)} \mathrm{div}(X_i)\,d\mathrm{vol}_g>0.
\end{equation}

By \eqref{eqn:div-almostnonnega}, Lemma~\ref{lma:GreeneWu}, Stoke Theorem, $0\leq T\leq \mathrm{scal}$, and the choice of $\varphi$, for any $r\geq 1$ and $i\in\mathbb{N}$ we have
\begin{equation}\label{eqn:inte-with err}
\begin{split}
&\quad \int_{B_g(x_0,1)} \left[\e_i\left( \frac{n}{2}-1\right)  \cdot \phi(b_i) \cdot \mathrm{scal}+\mathrm{div}(X_i)\right] \,d\mathrm{vol}_g\\
&\leq\int_{M} \left[\e_i\left( \frac{n}{2}-1\right)  \cdot \phi(b_i) \cdot \mathrm{scal}+\mathrm{div}(X_i)\right]\varphi \,d\mathrm{vol}_g \\
&=\e_i\left( \frac{n}{2}-1\right) \cdot \int_{M} \phi(b_i)\cdot \mathrm{scal}\cdot\varphi\,d\mathrm{vol}_g -\int_M \langle X_i,\nabla \varphi\rangle\,d\mathrm{vol}_g\\
&\leq \e_i\left( \frac{n}{2}-1\right) \cdot \int_{M} \phi(b_i)\cdot \mathrm{scal}\cdot\varphi\,d\mathrm{vol}_g +\frac{4}{r}\int_{A_g(x_0,r,2r)} \mathrm{scal}\,d\mathrm{vol}_g.
    \end{split}
\end{equation}

The conclusion follows by letting $i\to+\infty$ for each fixed $r\geq 1$, by our choice of $c_0>0$. 
\end{proof}

Finally we note that in case of $\mathrm{sec}>0$ nearby $x_0$, the constant $c_0$ in Theorem~\ref{thm:scalar-growth} can be estimated from below by curvature lower bound over ball centered at $x_0$. 
\begin{prop}\label{prop:c_1-local}
In the setting of Theorem~\ref{thm:scalar-growth}, if in addition $\mathrm{sec}>0$ on $B_g(x_0,r)$ for $r\in (0,1]$, then we might choose $c_1$ so that 
\begin{equation}
    c_1\geq \frac{(n-1)(n-2)}{8(1+r^2)^{3/2}}\cdot \mathrm{Vol}_g(x_0,r)\cdot \inf_{B_g(x_0,r)}\mathrm{sec}(g)>0.
\end{equation}
\end{prop}
\begin{proof}
We denote $\e:=\inf_{B_g(x_0,r)}\mathrm{sec}>0$.  In the proof of Theorem~\ref{thm:scalar-growth}, we might choose the Busemann function $b$ such that $b(x_0)=0=\inf_M b$. By the proof of \cite[Lemma 3.2]{ChanLee2025}, 
    \begin{equation}
        T_{ij}\geq  \frac12\e(n-1)(n-2)g_{ij}>0
    \end{equation}
    on $B_g(x_0,r)$ so that \eqref{eqn:div-almostnonnega} implies 
    \begin{equation}
        \begin{split}
            \mathrm{div}(X_i)\geq -o(1)+\frac12 \e(n-1)(n-2)\phi''(b_i) |\nabla b_i|^2\geq -o(1)
        \end{split}
    \end{equation}
on $B_g(x_0,r)$, as $i\to+\infty$. Furthermore by Lemma~\ref{lma:GreeneWu}, $b_i$ are $1$-Lipschitz and hence $b_i(x)\leq b_i(x_0)+r\leq \e_i+r$ on $B_g(x_0,r)$. In particular, 
\begin{equation}
    \phi''(b_i)=(1+b_i^2)^{-3/2}\geq \left(1+ (\e_i+r)^2\right)^{-3/2}.
\end{equation}

From the choice of $c_0$ in \eqref{eqn:choice-c_0}, this yields
\begin{equation}
    \begin{split}
 c_0&\geq \liminf_{i\to+\infty}       \int_{B_g(x_0,r)}  \mathrm{div}(X_i)\,d\mathrm{vol}_g\\
 &\geq \frac{\e(n-1)(n-2)}{2(1+r^2)^{3/2}}\cdot  \liminf_{i\to+\infty}   \int_{B_g(x_0,r)}  |\nabla b_i|^2\,d\mathrm{vol}_g\\
 &=\frac{\e(n-1)(n-2)}{2(1+r^2)^{3/2}}\cdot \mathrm{Vol}_g(x_0,r),
    \end{split}
\end{equation}
since $\nabla b_i$ converges weakly to $\nabla b$ and $|\nabla b|=1$ almost everywhere.     
\end{proof}

\section{Volume and Gap Theorem of three-manifolds under Ricci}

In this section, we discuss the proof of Theorem~
\ref{thm:gap}, i.e. how vanishing of asymptotic curvature forcing $M$ to be flat under $\Ric\geq 0$. We will prove a slightly more general Theorem, which says that if $k(x_0,r)=O(1)$ as $r\to+\infty$, then it is either $\mathrm{AVR}>0$ or have quadratic volume growth with finite total curvature. This is analogous to Corollary~\ref{cor:vol-decay-dich}. 
\begin{thm}\label{thm:vol-growth}
    Suppose $(M^3,g)$ is a complete non-compact non-flat manifold such that $\Ric\geq 0$, $\mathrm{AVR}(g)= 0$ and there exists $x_0\in M$ so that $\mathrm{scal}(x_0)>0$ and  $\b:=\limsup_{r\to+\infty} k(x_0,r)<+\infty$ where $k(x_0,r)$ is defined in \eqref{eqn:k}, then there is $\Lambda>1$ such that 
    \begin{enumerate}
        \item[(i)] $\Lambda^{-1}r^2\leq \mathrm{Vol}(x_0,r)\leq \Lambda r^2$ for all large $r>0$ and;
        \item[(ii)] $\displaystyle \int_M\mathrm{scal}(g)\,d\mathrm{vol}_g\leq \Lambda\b$.
    \end{enumerate}
\end{thm}

We note that the lim $\b$ is indeed independent of the base point $x_0$. Assuming Theorem~\ref{thm:vol-growth}, the gap Theorem~\ref{thm:gap} is immediate.
\begin{proof}[Proof of Theorem~\ref{thm:gap} assuming Theorem~\ref{thm:vol-growth}]

Suppose $\mathrm{AVR}>0$, then Proposition~\ref{prop:MV-inte-R} implies $\mathrm{AVR}=1$ and hence $(M^3,g)$ is flat by rigidity in volume comparison. Hence, we might assume $\mathrm{AVR}=0$. In this case, Theorem~\ref{thm:vol-growth} implies that total scalar curvature vanishes and thus flat. 
\end{proof}

In the remaining of this section, we are devoted to prove Theorem~\ref{thm:vol-growth}. For a later purpose, we also choose $\sigma\in (0,1)$ sufficiently small so that 
\begin{equation}\label{smallballmc}
    \int_{\partial B_g(x_0,\sigma)}H^2\,dA_g<16\pi.
\end{equation}
using  \cite[Proposition 3.1]{Mondino}.

As explained briefly in the introduction, we will start with observing that $(M^3,g)$ in Theorem~\ref{thm:vol-growth} must be at least quadratic volume growth.
\begin{lma}\label{lma:quad-vol-k}
Suppose $(M^n,g)$ is a complete non-flat non-compact manifold such that $\limsup_{r\to+\infty} k(x_0,r)<+\infty$, then there is $c_1>0$ such that 
\begin{equation}
\liminf_{r\to+\infty}\frac{\mathrm{Vol}_g(x_0,r)}{r^2}\geq c_1.
\end{equation}
\end{lma}
\begin{proof}
Since $k(x_0,r)$ is bounded from above as $r\to+\infty$,  we choose $R_0, C_0>1$ such that for all $r>R_0$, $ k(x_0,r)\leq C_0$ and hence
\begin{equation}
\begin{split}
    \mathrm{Vol}_g(x_0,r) &\geq C_0^{-1} r^2 \cdot\int_{B_g(x_0,r)}\mathrm{scal}\,d\mathrm{vol}_g\\
    &\geq C_0^{-1} r^2 \cdot\int_{B_g(x_0,1)} \mathrm{scal}\,d\mathrm{vol}_g,
\end{split}
\end{equation}
for all $r>R_0$. Since $\mathrm{scal}(x_0)>0$, result follows.
\end{proof}

We note that Lemma~\ref{lma:quad-vol-k} is dimension independent but for convenience, we restrict to $n=3$. Lemma~\ref{lma:quad-vol-k} implies that $(M^3,g)$ is $p$-nonparabolic for a well-chosen $p\in (1,2]$. 
\begin{lma}\label{lma:nonparabolic}
$(M^3,g)$ is $p$-nonparabolic, for $p=4/3\in (1,2)$.
\end{lma}
\begin{proof}
By Lemma~\ref{lma:quad-vol-k}, for any $x\in M$ we have 
\begin{equation}
    \begin{split}
        \int_1^\infty\left(\frac{s}{\mathrm{Vol}_g(x,s)}\right)^{\frac{1}{\frac{4}{3}-1}}ds&=\int_1^\infty\left(\frac{s}{\mathrm{Vol}_g(x,s)}\right)^3ds\\
        &\lesssim\int_1^\infty\frac{1}{s^3}\,ds<+\infty.
    \end{split}
\end{equation}

Since $\Ric\geq 0$, the assertion follows by \cite[Proposition 5.10]{Holopainen1999}, see \cite[Remark 4.2]{FM2022}.
\end{proof}

We consider the Dirichlet problem of the $p$-harmonic function on the exterior domain $M\setminus \partial B_g(x_0,\sigma)$:
\begin{equation}\label{eqn:p-harm}
    \left\{
    \begin{array}{cc}
         \Delta_{\frac{4}{3}}u=0 \quad &\textit{in  } \Omega:=M\setminus \overline{B_g(x_0,\sigma)}\\[2mm]
          u=1\quad&\text{on  } \partial B_g(x_0,\sigma)\\[2mm]
    u(x)\to 0\quad&\text{as  }  x\to \infty.
    \end{array}
    \right.
\end{equation}

\begin{lma}\label{lma:p-exist-estimate}
Under the assumption of Theorem~\ref{thm:vol-growth}, \eqref{eqn:p-harm} admits a solution $u$ such that 
\begin{enumerate}
    \item [(i)] $0<u<1$ on $M\setminus \overline{B_g(x_0,\sigma)}$;
    \item[(ii)] $\int_{\Omega} |\nabla u|^{4/3}\,d\mathrm{vol}_g <+\infty$;
    \item[(iii)] There is $C_0>0$ such that for $x\to\infty$, 
    \begin{equation}
        u(x)\leq C_0\left(\int_{d_g(x,x_0)}^\infty\mathrm{Vol}(x_0,s)^{-3/4}\,d s\right)^4;
    \end{equation}
    \item[(iv)] $u$ is $C^{1,\b}$ up to $\partial \Omega$ and is smooth with non-vanishing gradient in a neighborhood of $\partial\Omega$.
\end{enumerate}
\end{lma}
\begin{proof}
The existence of $u$ with (ii) follows from  \cite[Theorem 4.1]{FM2022} and Lemma~\ref{lma:nonparabolic}. The property (i) follows from strong maximum principle of $p$-harmonic function. The estimate (iii) for $p$-Green function follows from \cite[Proposition 4.4]{KotschwarNi} and \cite[Proposition 5.7]{Ho}, as pointed out in \cite[p. 15]{KotschwarNi}, it carries over to Dirichlet solution $u$. The smoothness across the boundary follows from \cite[Theorem 4.1 \& Appendix A]{FM2022}. 
\end{proof}

It is more convenient to consider the normalized function. We denote $\rho:=d_g(x_0,\cdot)$, $ w:=-\frac13\log u$ and
\begin{equation}\label{zg:w}
 \Omega_t:=\overline{B_g(x_0,\sigma)}\cup
           \{x\in M\setminus\overline{B_g(x_0,\sigma)}:w(x)<t\}, \, t>0.
\end{equation}

\begin{lma}[Boundary regularity and the logarithmic gradient]
\label{lma:gradient}
We have $w>0$ on $\Omega$, vanishes on $\partial \Omega$ and satisfies
\[
 \operatorname{div}(|\nabla w|^{-2/3}\nabla w)=|\nabla w|^{4/3}.
\]
For each $T>0$, the set $\{x\in M\setminus\Omega:0\le w(x)\le T\}$
is compact. There is $t_0>0$ such that $w$ is smooth and $|\nabla w|>0$
throughout this set for $T=t_0$. Moreover,
\[
 \sup_{M\setminus\Omega}|\nabla w|<\infty,
 \qquad |\nabla w|(x)\le\frac{C}{\rho(x)}\quad(\rho(x)\ge2\sigma).
\]
\end{lma}
\begin{proof}
By Lemma~\ref{lma:p-exist-estimate}, there is $\delta>0$ such that $w$ is smooth and $|\nabla w|>0$ on
$\{\sigma\le\rho\le\sigma+\delta\}$.
Since $u<1$ in the exterior and $u\to0$ at infinity,
$\sup_{\{\rho\ge\sigma+\delta\}}u<1$.
Choose $t_0>0$ with
$e^{-3t_0}>\sup_{\{\rho\ge\sigma+\delta\}}u$. Then
\begin{equation}\label{zg:initial-level-location}
 \{x\in M\setminus B_g(x_0,\sigma):0\le w(x)\le t_0\}
 \subset\{x:\sigma\le\rho(x)<\sigma+\delta\}.
\end{equation}
The same decay makes every $\{0\le w\le T\}$ in the closed exterior
closed and bounded, hence compact. Also $w=0$ on $\partial B_g(x_0,\sigma)$
and $w>0$ outside.

 Direct computation gives
$\operatorname{div}(|\nabla w|^{-2/3}\nabla w)=|\nabla w|^{4/3}$.
For $\rho(x)\ge2\sigma$, the ball $B_{\rho(x)/4}(x)$ lies in the
exterior. \cite[Theorem 1.1]{WZ2012} therefore gives
\begin{equation}\label{zg:gradient}
 |\nabla w|(x)=\tfrac13|\nabla\log u(x)|
 \le\frac{4C_W}{3\rho(x)}\qquad(\rho(x)\ge2\sigma),
\end{equation}
where $C_W$ depends only on dimension and $p=4/3$.
On the remaining compact annulus, interior and boundary $C^1$
regularity and positivity imply
\begin{equation}\label{zg:gradient-compact}
 \sup_{\{\sigma\le\rho\le2\sigma\}}|\nabla w|
 \le
 \frac{\displaystyle\max_{\{\sigma\le\rho\le2\sigma\}}|\nabla u|}
      {3\displaystyle\min_{\{\sigma\le\rho\le2\sigma\}}u}<\infty.
\end{equation}
Together these estimates give the global bound for $|\nabla w|$.
\end{proof}

Denote the regular part of each level set of $w$ by $\Sigma_t$:
\begin{equation}
    \Sigma_t:=\{x\in M\setminus B_g(x_0,\sigma): w(x)=t, \nabla w\neq 0.\}.
\end{equation}
so that 
\begin{equation}\label{GBF}
    \int_{\Sigma_t}\mathrm{scal}(g_{\Sigma_t})\in 8\pi\mathbb{Z},
\end{equation}
for  almost all positive $t$, by \cite[Theorem 1.3]{BPP2024}.

We consider the following quantities similar to \cite{BQOP2024}
\begin{equation}
    G(t):=\frac{9}{25}\int_{\Sigma_t}|\nabla w|^2, \quad F(t):=\frac{3}{5}\int_{\Sigma_t}H|\nabla w|-G(t),
\end{equation}
in which
\[
F(t)=\int_{\Sigma_t} \frac{H^2}{4}-\left(\frac H2-\frac{3}{5}|\nabla w|\right)^2\leq \int_{\Sigma_t} \frac{H^2}{4}.
\]
From \eqref{smallballmc}, we also have $F(0)<4\pi$. More importantly, $F(t)$ and $G(t)$ satisfy a monotonicity formula under $\Ric\geq 0$. 

\begin{lma}\label{FGlema}\cite[Lemmas 2.2 and 2.3]{BQOP2024}
    $F\in W^{1,1}_{loc}(0,\infty)$ and $G\in W^{2,1}_{loc}(0,\infty)$ are continuous at $t=0$ and satisfy the following equations for a.e. $t>0$
\begin{eqnarray}
 G'(t)&=&3(G(t)-F(t))\\
F'(t)&=&-\frac{3}{5}\int_{\Sigma_t}\left[\Ric(\nu,\nu)+|A^\circ|^2+|\nabla^{\Sigma_t}\ln |\nabla w||^2+\frac{5}{2}\left(H-\frac{6}{5}|\nabla w|\right)^2\right].
\end{eqnarray}
Both $F$ and $G$ are nonincreasing in $t$,
\[
0\leq G(t)\leq F(t)\leq F(0)<4\pi\qquad(t\ge0)
\]
\begin{equation}\label{zg:square-total}
 \int_0^\infty\left(\int_{\Sigma_t}
       \left(H-\frac65|\nabla w|\right)^2 \right)\, dt
 \le\frac23F(0).
\end{equation}
\end{lma}
\begin{proof}
The proof can be found in \cite[Lemmas 2.2 and 2.3]{BQOP2024}. The integral estimates follows from integrating $F'$ from $t=0$ to $t=+\infty$ using $\Ric\geq 0$ and $F(0)<4\pi$.
\end{proof}

\begin{lma}\label{lma:flux}\cite{BFM2024}
For almost every $t>0$,
\begin{equation}\label{zg:flux}
 e^{-t}\int_{\Sigma_t}|\nabla w|^{1/3}
 =\int_{\partial B_g(x_0,\sigma)}|\nabla w|^{1/3}
      =3^{-1/3}\int_{\partial B_g(x_0,\sigma)}|\nabla u|^{1/3}>0.
\end{equation}
Moreover,  $F,G$ extend absolutely continuously to every
$[0,T]$, with their values at $t=0$ equal to the 
integrals over $\partial B_g(x_0,\sigma)$. 
\end{lma}
\begin{proof}
For the potential $u$ constructed,
\[
 |\nabla u|^{-2/3}\nabla u
       =-3^{1/3}e^{-w}|\nabla w|^{-2/3}\nabla w,
\]
with value zero whenever $\nabla w=0$. \cite[Proposition 2.8]{BFM2024} gives, for almost every $\tau\in(0,1)$,
\[
 \int_{\{u=\tau\}}|\nabla u|^{1/3}
 =\int_{\partial B_g(x_0,\sigma)}|\nabla u|^{1/3}.
\]
On $\Sigma_t$, $u=e^{-3t}$ and $|\nabla u|^{1/3}=3^{1/3}e^{-t}|\nabla w|^{1/3}$. This proves the identity. The strict positivity then follows from (iv) of Lemma~\ref{lma:p-exist-estimate}. Note that Lemma \ref{lma:gradient} implies that the level set $\Sigma_t$ of $w$ are regular for $t\in [0,t_0]$. Hence $F$ and $G$ are smooth for $t\in [0,t_0]$, together with $F\in W^{1,1}_{loc}(0,\infty)$ and $G\in W^{2,1}_{loc}(0,\infty)$ from Lemma \ref{FGlema}, $F,G$ extend absolutely continuously to every $[0,T]$. 
\end{proof}

\medskip

The next Proposition give us a control of volume lower bound through sub-level set of potential function.

\begin{prop}\label{prop:sublevels}
There are constants $c,C>0$ such that for all $T>1$,
\begin{equation}\label{zg:volume-sublevel}
 \mathrm{Vol}_g(\Omega_T)\ge c\exp\left[3T-
       C\int_{\Omega_T\setminus B_g(x_0,\sigma)}\mathrm{scal}(g)|\nabla w| \,d\mathrm{vol}_{g}\right]
 \qquad(T>1).
\end{equation}
After increasing the fixed radius $r_0$ if necessary, there are
constants $c,C>0$, independent of $r,T$, such that if $T>1,\ r\ge r_0,\ \Omega_T\subset B_g(x_0,r)$
\begin{equation}\label{zg:volume-ball-comparison}
 \mathrm{Vol}_g(x_0,r)\ge c\exp\left[3T-
       C\int_{B_g(x_0,r)\setminus B_g(x_0,r_0)}\frac{\mathrm{scal}(g)}{\rho} \,d\mathrm{vol}_{g}\right].
\end{equation}
\end{prop}

\begin{proof}
By Lemma~\ref{FGlema}, $F-G/2\ge G/2\ge 0$. Thanks to Lemma \ref{lma:flux} and monotonicity from Lemma~\ref{FGlema}, $G(t)>0$ for all $t$ and  thus $F-G/2>0$. 

By Lemma~\ref{lma:flux}, $\log(F-G/2)$ is absolutely continuous.
We first obtain a differential inequality for this logarithm. The Gauss equation and the definition of $F$ give 
\begin{equation}\label{zg:gauss-recalled}
\begin{aligned}
 \int_{\Sigma_t}\left[\mathrm{scal}(g)+\frac12\left(H-\frac65|\nabla w|\right)^2\right]
 &=\int_{\Sigma_t}\mathrm{scal}(g_{\Sigma_t})-2F(t)\\
 &\quad+\int_{\Sigma_t}\bigl[2\Ric(\nu,\nu)+|A^\circ|^2\bigr].
\end{aligned}
\end{equation}

By Cauchy--Schwarz, we have 
\begin{equation}\label{zg:cauchy-square}
 (F-G)^2
 =\left[\frac35\int_{\Sigma_t}|\nabla w|
                 \left(H-\frac65|\nabla w|\right)\right]^2
 \le G\int_{\Sigma_t}\left(H-\frac65|\nabla w|\right)^2.
\end{equation}

\begin{claim}\label{claim:inequal-FG}
\begin{equation}\label{zg:log-ode}
\begin{split}
 \frac{d}{dt}\log\left(F-\frac12G\right)
 +\frac32 \le\frac{3}{8\pi-2F(0)}
 \int_{\Sigma_t}\left[\mathrm{scal}(g)+\frac12\left(H-\frac65|\nabla w|\right)^2\right].
 \end{split}
\end{equation}
\end{claim}
\begin{proof}[Proof of Claim~\ref{claim:inequal-FG}]

By \eqref{GBF}, for almost every $t$ either one of the following holds:
\begin{enumerate}
    \item [(I)]$\displaystyle\int_{\Sigma_t}\mathrm{scal}(g_{\Sigma_t})\leq 0$ or
    \item [(II)] $\displaystyle\int_{\Sigma_t}\mathrm{scal}(g_{\Sigma_t})\geq 8\pi $.
\end{enumerate}

\medskip

In case of $(\mathrm{I})$,
\eqref{zg:gauss-recalled} and $\mathrm{scal}(g)\ge\Ric(\nu,\nu)$ imply
\[
 \int_{\Sigma_t}\bigl[\Ric(\nu,\nu)+|A^\circ|^2\bigr]
 \ge2F+\frac12\int_{\Sigma_t}\left(H-\frac65|\nabla w|\right)^2.
\]

Substituting into the formula of $F'$ in Lemma \ref{FGlema}
gives 
$$F'\le-\frac65F-\frac95\int_{\Sigma_t}(H-\frac65|\nabla w|)^2.$$

Thus $G'=3(G-F)$ and \eqref{zg:cauchy-square} yield
\begin{align*}
 &\quad \left(F-\frac12G\right)'+\frac32\left(F-\frac12G\right)\\
 &=F'+3F-\frac94G\\
 &\le\frac95F-\frac94G-\frac95\frac{(F-G)^2}{G}
 =-\frac95\frac{(F-\frac32G)^2}{G}\le0.
\end{align*}
This proves the inequality in case $(\mathrm{I})$.

\medskip

In case of $(\mathrm{II})$, 
\eqref{zg:gauss-recalled} and $F\le F(0)<4\pi$ give
\begin{equation}\label{eqn:S-H^2}
    \int_{\Sigma_t}\left[\mathrm{scal}(g)+\frac12\left(H-\frac65|\nabla w|\right)^2\right]
 \ge8\pi-2F(0)>0.
\end{equation}

In this case it suffices to use $F'\le0$ and $G\ge0$ to show that 
\[
 \frac{d}{dt}\log\left(F-\frac12G\right)+\frac32
 =\frac{F'+3F-\frac94G}{F-\frac12G}\le3.
\]
Together with \eqref{eqn:S-H^2} finish the proof.
\end{proof}

\medskip

We integrate $\log (F-\frac12 G)$ to conclude 
\begin{equation}
    \begin{split}
     &\quad   \left[ \log\left( F-\frac12 G\right) +\frac32 \right] \Big|^t_0 \\&\leq\frac{3}{8\pi-2F(0)} \int^t_0 
 \int_{\Sigma_s}\left[\mathrm{scal}(g)+\frac12\left(H-\frac65|\nabla w|\right)^2\right]ds\\
 &=:\mathbf{I}+\mathbf{II}.
    \end{split}
\end{equation}

For $\mathbf{I}$, we use coarea to give 
\begin{equation}
    \begin{split}
\int^t_0\left(\int_{\Sigma_s} \mathrm{scal}(g) \right)ds =\int_{\Omega_t\setminus B_g(x_0,\sigma)} \mathrm{scal}(g)\cdot |\nabla w|\,d\mathrm{vol}_g
    \end{split}
\end{equation}
while
\begin{equation}
    \begin{split}
\int^t_0\left[\int_{\Sigma_s} \left(H-\frac65 |\nabla w| \right)^2 \right]ds\leq \frac23 F(0),
    \end{split}
\end{equation}
using Lemma~\ref{FGlema}. Exponentiating and using $G\le2(F-G/2)$ now gives
\begin{equation}\label{zg:G-exponential}
 G(t)\le C_0\exp\left[-\frac32t+
 C_0\int_{\Omega_t\setminus B_g(x_0,\sigma)}\mathrm{scal}(g)|\nabla w|\,d\mathrm{vol}_{g}\right].
\end{equation}

To pass to volume, H\"older's inequality on the regular part of $\Sigma_t$ gives
\[
 \int_{\Sigma_t}|\nabla w|^{1/3}
 \le\left(\int_{\Sigma_t}|\nabla w|^2\right)^{4/9}
     \left(\int_{\Sigma_t}|\nabla w|^{-1}\right)^{5/9}.
\]
The flux formula \eqref{zg:flux} makes the left side
$e^t\int_{\partial B_g(x_0,\sigma)}|\nabla w|^{1/3}$, and
$\int_{\Sigma_t}|\nabla w|^2=(25/9)G(t)$. Rearranging and using
\eqref{zg:G-exponential}, we obtain
\begin{equation}\label{zg:area-inverse-gradient}
\begin{aligned}
 \int_{\Sigma_t}|\nabla w|^{-1}
 &\ge c e^{9t/5}G(t)^{-4/5}\\
 &\ge c\exp\left[3t-
 C\int_{\Omega_t\setminus B_g(x_0,\sigma)}\mathrm{scal}(g)|\nabla w|\,d\mathrm{vol}_{g}\right].
\end{aligned}
\end{equation}

The fixed boundary flux is positive; the exponent is
$9/5+(4/5)(3/2)=3$. If the left side is infinite, the inequality is automatic.
Coarea on $\{|\nabla w|>0\}$ and $\Omega_t\subset\Omega_T$ for $t<T$ give
\begin{align*}
 \mathrm{Vol}_g(\Omega_T)
 &\ge\int_{T-1}^T\left[\int_{\Sigma_t}|\nabla w|^{-1}\right]\,d t\\
 &\ge c\exp\left[-C\int_{\Omega_T\setminus B_g(x_0,\sigma)}
                      \mathrm{scal}(g)|\nabla w|\,d\mathrm{vol}_{g}\right]\int_{T-1}^Te^{3t}\,d t.
\end{align*}

Since $\int_{T-1}^Te^{3t}\,d t=(1-e^{-3})e^{3T}/3$, this proves
\eqref{zg:volume-sublevel}. This argument discards any volume of the critical
set.

Finally, choose $r_0\ge2\sigma$ so that $|\nabla w|\le C/\rho$ for
$\rho\ge r_0$, using Lemma~\ref{lma:gradient}. If $\Omega_T\subset B_g(x_0,r)$,
\begin{eqnarray*}
     \int_{\Omega_T\setminus B_g(x_0,\sigma)}\mathrm{scal}(g)|\nabla w|\,d\mathrm{vol}_{g}
 &\le&\left(\sup_{\{\sigma\le\rho\le r_0\}}|\nabla w|\right)
       \int_{B_g(x_0,r_0)}\mathrm{scal}(g)\,d\mathrm{vol}_{g}\\
&&       +C\int_{B_g(x_0,r)\setminus B_g(x_0,r_0)}\frac {\mathrm{scal}(g)}{\rho}\,d\mathrm{vol}_{g}.
\end{eqnarray*}

The first term is a fixed finite constant by Lemma~\ref{lma:gradient}
and smoothness of $\mathrm{scal}(g)$ on the compact closed ball. Absorbing it into $c$
in \eqref{zg:volume-sublevel} proves \eqref{zg:volume-ball-comparison}.
\end{proof}

\medskip

For $r_0$ from Proposition~\ref{prop:sublevels}, we define 
\begin{equation}\label{zg:m-definition}
 m(r):=\inf_{s\ge r}\frac{\mathrm{Vol}_g(x_0,s)}{s^2},\;\;\forall r\geq r_0.
\end{equation}

By Lemma~\ref{lma:quad-vol-k}, we might assume $r_0$ large so that $k(x_0,r)\leq C_1:=\b+1$ for $r>r_0$ and  
\begin{equation}\label{mmono}
    0<c\le m(r)\le\frac{\mathrm{Vol}_g(x_0,r)}{r^2},
 \qquad m(r_1)\le m(r_2)
\end{equation}
for $r_2\geq r_1\geq r_0$. Moreover by $\mathrm{scal}\geq 0$, for all $s>r>r_0$ 
\begin{equation}\label{eqn:mon-sca}
    \begin{split}
        \int_{B_g(x_0,r)} \mathrm{scal}\,d\mathrm{vol}_g&\leq   \int_{B_g(x_0,s)} \mathrm{scal}\,d\mathrm{vol}_g\\
        &=k(x_0,s)\cdot \frac{\mathrm{Vol}_g(x_0,s)}{s^2}
    \end{split}
\end{equation}

Taking the infimum over $s>r$ shows that for all $r\geq r_0$,
\begin{equation}\label{di:scalar-m}
 \int_{B_g(x_0,r)}\mathrm{scal}(g)\,d\mathrm{vol}_{g}\le C_1m(r).
\end{equation}

Since $\mathrm{AVR}(g)=0$, by \eqref{eqn:mon-sca} and Bishop-Gromov comparison for any $\e>0$, there is $r_\e>r_0$ such that for all $r>r_\e$,
\begin{equation}\label{di:scalar-m-2}
\frac1r \int_{B_g(x_0,r)}\mathrm{scal}(g)\,d\mathrm{vol}_{g}\le \e .
\end{equation}

\medskip

We next show that the error term appeared in Proposition~\ref{prop:sublevels} is negligible if $m(r)\to+\infty$.
\begin{lma}\label{lem:weighted}
Under the  assumptions of Theorem~\ref{thm:vol-growth}, if in addition $m(r)\to\infty$, then
\[
 \frac{1}{\log m(r)}
 \int_{B_g(x_0,r)\setminus B_g(x_0,r_0)}\rho^{-1}\cdot \mathrm{scal}(g)\,d\mathrm{vol}_{g}\longrightarrow0.
\]
\end{lma}
\begin{proof}
By assumption, $\log m(r)>0$ for all large $r$. The hypothesis $m(r)\to\infty$ ensures that $\log m(r)>0$ for all
sufficiently large $r$. For $r_0<\rho<r$, the identity
$1/\rho=1/r+\int_\rho^r s^{-2}\,ds$ and Tonelli's theorem give
\begin{equation}\label{zg:tonelli-scalar}
\begin{split}
&\quad  \int_{B_g(x_0,r)\setminus B_g(x_0,r_0)}\frac{\mathrm{scal}(g)}{\rho}\,d\mathrm{vol}_{g}\\
& \le\frac1r\int_{B_g(x_0,r)}\mathrm{scal}(g)\,d\mathrm{vol}_g      +\int_{r_0}^r\frac1{s^2}\left(\int_{B_g(x_0,s)}\mathrm{scal}(g)\,d\mathrm{vol}_{g}\right)ds.
\end{split}
\end{equation}

The first term is bounded by \eqref{di:scalar-m-2}. For the second, we use \eqref{di:scalar-m} instead that 
\begin{equation}
    \begin{split}
       &\quad \int^r_{r_0} \frac1{s^2}\left(\int_{B_g(x_0,s)}\mathrm{scal}(g)\,d\mathrm{vol}_{g}\right)ds\\
       &\leq \int^{r_\e}_{r_0}\frac1{s^2}\left(\int_{B_g(x_0,s)}\mathrm{scal}(g)\,d\mathrm{vol}_{g}\right)ds+\int^r_{r_\e} s^{-2}\cdot \min\{C_1m(s), \e s\}\,ds
    \end{split}
\end{equation}

We compute the second term explicitly that 
\begin{equation}
    \begin{split}
 \int^r_{r_\e} s^{-2}\cdot \min\{C_1m(s), \e s\}\,ds
 &\leq \int_{r_\e}^{A_0 m(r)/\e}\frac{\e}{s}\,ds
       +\int_{C_1 m(r)/\e}^{\infty}\frac{C_1m(s)}{s^2}\,ds\\
 &=\e\log\frac{C_1 m(r)}{\e r_\e}+\e.
    \end{split}
\end{equation}

Result follows from first letting $r\to+\infty$ and followed by $\e\to0$.
\end{proof}

\medskip

We now prove that the total curvature is bounded.
\begin{prop}\label{prop:finite}
Under the assumption of Theorem~\ref{thm:vol-growth}, we have 
\[
 \sup_{r>0}m(r)<\infty,
 \qquad 0<\int_M\mathrm{scal}(g)\,d\mathrm{vol}_{g}<\infty.
\]
\end{prop}
\begin{proof}

By \eqref{di:scalar-m},  it suffices to bound $m$. If $m$ were
unbounded, then \eqref{mmono} would give
\begin{equation}\label{di:m-diverges}
 m(r)\longrightarrow\infty,
 \qquad \frac{\mathrm{Vol}_g(x_0,r)}{r^2}\ge m(r)\longrightarrow\infty.
\end{equation}

We now derive a contradiction
using Proposition~\ref{prop:sublevels}. By Lemma~\ref{lma:p-exist-estimate}
for $\rho(x)\to+\infty$,
\begin{align*}
 u(x)
 &\le C\left(\int_{\rho(x)}^\infty\frac{\,d s}{\mathrm{Vol}_g(x_0,s)^{3/4}}\right)^4\\
 &\le \frac{C}{m(\rho(x))^3}
             \left(\int_{\rho(x)}^\infty s^{-3/2}\,d s\right)^4\\
 &=\frac{16C}{m(\rho(x))^3\rho(x)^2}.
\end{align*}
The second line uses $\mathrm{Vol}_g(x_0,s)\ge m(\rho(x))s^2$ for every
$s\ge\rho(x)$. This implies
\begin{equation}\label{zg:potential-barrier}
 w(x)\ge\frac23\log\rho(x)+\log m(\rho(x))-C_0.
\end{equation}

The continuous function $r\mapsto \mathrm{Vol}_g(x_0,r)/r^2$ tends to infinity, so
it attains a minimum on every closed interval $[j,\infty)$ with
integer $j\ge r_0$. To see this
directly, its values eventually exceed $\mathrm{Vol}_g(x_0,j)/j^2+1$; a minimum on
the remaining compact interval is therefore also a minimum on the
whole tail. Choose a minimizing point $r_j\ge j$. For every
$s\ge r_j$ we then have $\mathrm{Vol}_g(x_0,s)/s^2\ge \mathrm{Vol}_g(x_0,r_j)/r_j^2$, and hence
\begin{equation}\label{zg:minimizing-radii}
 \frac{\mathrm{Vol}_g(x_0,r_j)}{r_j^2}=m(r_j),\qquad
 r_j\longrightarrow\infty,\qquad m(r_j)\longrightarrow\infty.
\end{equation}
Define the associated levels by
\[
 T_j=\frac23\log r_j+\log m(r_j)-C_0-1.
\]
Equations~\eqref{zg:minimizing-radii} give $T_j\to\infty$.
For all large $j$, \eqref{mmono} and
\eqref{zg:potential-barrier} imply
\[
 \rho(x)\ge r_j
 \quad\Longrightarrow\quad
 w(x)\ge\frac23\log r_j+\log m(r_j)-C_0=T_j+1.
\]
Thus $\Omega_{T_j}\subseteq B_{g}(x_0,r_j)$. Since
\[
 e^{3T_j}=e^{-3C_0-3}r_j^2m(r_j)^3,
\]
substitution into \eqref{zg:volume-ball-comparison} gives
\[
 r_j^2m(r_j)=\mathrm{Vol}_g(x_0,r_j)
 \ge c\,r_j^2m(r_j)^3
       \exp\left[-C\int_{B_g(x_0,r_j)\setminus B_g(x_0,r_0)}
                              \frac{\mathrm{scal}(g)}{\rho}\,d\mathrm{vol}_{g}\right].
\]

Rearrangement gives
\[
 2\log m(r_j)
 \le C+C\int_{B_g(x_0,r_j)\setminus B_g(x_0,r_0)}
                              \frac{\mathrm{scal}(g)}{\rho}\,d\mathrm{vol}_{g}.
\]
and hence,
\[
 2\le\frac{C}{\log m(r_j)}+
       \frac{C}{\log m(r_j)}
       \int_{B_g(x_0,r_j)\setminus B_g(x_0,r_0)}\frac {\mathrm{scal}(g)}{\rho}\,d\mathrm{vol}_{g}
       \longrightarrow0,
\]
by Lemma~\ref{lem:weighted}. This contradicts
\eqref{di:m-diverges} and thus $m$ is bounded. 
\end{proof}

Finally we show that the volume can at most growth quadratically. 
\begin{prop}
\label{prop:quadratic-upper}
Under the assumption of Theorem~\ref{thm:vol-growth}, we have 
\begin{equation}
    \limsup_{r\to+\infty}\frac{\mathrm{Vol}_g(x_0,r)}{r^2}<+\infty.
\end{equation}
\end{prop}
\begin{proof}

From \eqref{zg:gradient}, we already have
\[
 \sup_{M\setminus\Omega}|\nabla w|<\infty.
\]
Together with the assumed finite scalar integral, this implies
\[
 \int_{M\setminus\Omega}\mathrm{scal}(g)|\nabla w|\,d\mathrm{vol}_{g}
 \le\left(\sup_{M\setminus\Omega}|\nabla w|\right)
         \int_M \mathrm{scal}(g)\,d\mathrm{vol}_{g}<\infty
\]
by Proposition~\ref{prop:finite}.

Substituting this uniform bound into \eqref{zg:volume-sublevel} gives
\begin{equation}\label{av:finite-scalar-sublevel}
 \mathrm{Vol}_g(\Omega_T)\ge c e^{3T},
\end{equation}
for all $T\to+\infty$, where $c>0$.

From Lemma~\ref{lma:p-exist-estimate}, for all large $\rho(x)$ 
\begin{equation}\label{eqn:u-bbbd}
    u(x)\le C_0\left(\int_{\rho(x)}^\infty
                                 \frac{\,d s}{\mathrm{Vol}_g(x_0,s)^{3/4}}\right)^4.
\end{equation}

For a large radius $r$, set
\[
 T(r):=-\frac13\log\left[C_0
                   \left(\int_r^\infty\frac{\,d s}{\mathrm{Vol}_g(x_0,s)^{3/4}}\right)^4
              \right]-1
\]
so that Lemma~\ref{lma:quad-vol-k} implies $T\to\infty$ as
$r\to\infty$. 

Hence if $\rho(x)\geq r$, then $w(x)\geq T(r)+1$ by \eqref{eqn:u-bbbd}. This proves $\Omega_T\subset B_g(x_0,r)$ for large $r$. Applying
\eqref{av:finite-scalar-sublevel} gives
\[
 \mathrm{Vol}_g(x_0,r)\ge ce^{3T}
   =\frac{ce^{-3}}{C_0
                  \left(\displaystyle\int_r^\infty
                                \mathrm{Vol}_g(x_0,s)^{-3/4}\,d s\right)^4}.
\]
Thus there is a fixed $c_2>0$ such that
\begin{equation}\label{av:tail-product}
 \mathrm{Vol}_g(x_0,r)\cdot \left(\int_r^\infty\frac{\,d s}{\mathrm{Vol}_g(x_0,s)^{3/4}}\right)^4
 \ge c_2,
\end{equation}
for all $r$ large.

The fundamental theorem of calculus therefore gives
\begin{equation}
    \begin{split}
         \frac{d}{dr}
 \left(\int_r^\infty\frac{\,d s}{\mathrm{Vol}_g(x_0,s)^{3/4}}\right)^{-2}
 &=2\left[\mathrm{Vol}_g(x_0,r)
       \left(\int_r^\infty\frac{\,d s}{\mathrm{Vol}_g(x_0,s)^{3/4}}\right)^4
       \right]^{-3/4}\\
& \le2c_2^{-3/4}.
    \end{split}
\end{equation}

Choose a fixed $r_2\ge r_0$ so that \eqref{av:tail-product} holds
for $r\ge r_2$. Integration yields
\begin{equation}
    \begin{split}
    \left(\int_r^\infty\frac{\,d s}{\mathrm{Vol}_g(x_0,s)^{3/4}}\right)^{-2}
 &\le\left(\int_{r_2}^\infty\frac{\,d s}{\mathrm{Vol}_g(x_0,s)^{3/4}}\right)^{-2}
       +2c_2^{-3/4}(r-r_2)\\
 &\le\left[
       \frac1{r_2}\left(\int_{r_2}^\infty
                              \frac{\,d s}{\mathrm{Vol}_g(x_0,s)^{3/4}}\right)^{-2}
       +2c_2^{-3/4}\right]r
    \end{split}
\end{equation}
for $r\geq r_2$.
\medskip

That said for $r$ sufficiently large, we have 
\begin{equation}\label{av:tail-two-sided}
 \int_r^\infty\frac{\,d s}{\mathrm{Vol}_g(x_0,s)^{3/4}}
 \geq C^{-1}r^{-1/2}.
\end{equation}

By restricting $s\in [r,2r]$, this gives us the  desired volume upper bound of $\mathrm{Vol}_g(x_0,r)$.
\end{proof}

Now we have all ingredients for Theorem~\ref{thm:vol-growth}.
\begin{proof}[Proof of Theorem~\ref{thm:vol-growth}]

The volume lower bound follows from Lemma~\ref{lma:quad-vol-k}, the finiteness of total scalar curvature follows from Proposition~\ref{prop:finite}, while the quadratic volume upper bound follows from Proposition~\ref{prop:quadratic-upper}. It remains to show that the total scalar curvature is bounded by multiple of $\b:=\limsup_{r\to+\infty} k(x_0,r)$. It follows from refining \eqref{eqn:mon-sca}. For any $\e>0$, there is $r_\e>0$ such that for all $r_\e>r$, 
\begin{equation}
\begin{split}
 \int_{B_g(x_0,r)}\mathrm{scal}\,d\mathrm{vol}_g
 &\leq \int_{B_g(x_0,r_\e)}    \mathrm{scal}\,d\mathrm{vol}_g\\
 &\leq k(x_0,r_\e)\cdot \frac{\mathrm{Vol}_g(x_0,r_\e)}{r_\e^2}\leq (\b+\e)\cdot \Lambda.
\end{split}
\end{equation}
Result follows by letting $\e\to0$.

\end{proof}

\end{document}